\documentclass{amsart}
\usepackage{amsfonts}
\usepackage{latexsym,amssymb}
\usepackage{graphicx}

\usepackage[active]{srcltx}

\newtheorem{theorem}{Theorem}
 
 \newtheorem{lemma}[theorem]{Lemma}
 \newtheorem{proposition}[theorem]{Proposition}
\theoremstyle{definition}

  \DeclareMathOperator{\ii}{i} 
 \DeclareMathOperator{\ou}{ou} \DeclareMathOperator{\inn}{in}

 \DeclareMathOperator{\rev}{rev} 
\DeclareMathOperator{\tr}{tr} 
\DeclareMathOperator{\diag}{diag} \DeclareMathOperator{\supp}{supp}

\newcommand{\Real}{\mathbb{R}}
\newcommand{\Comp}{\mathbb{C}}
\newcommand{\Nat}{\mathbb{N}}

\newcommand{\Ee}{\mathbb{E}}

\newcommand{\eps}{\varepsilon}
\newcommand{\dts}{,\dots,}

\newcommand{\sbs}{\subset}

\newcommand{\set}[1]{\left\{#1\right\}}

\newcommand{\matp}[1]{\begin{bmatrix} #1 \end{bmatrix}}

\usepackage{color}

\begin{document}

   \title[  ]
   {On the spectral properties of a class of $H$-selfadjoint random matrices and the underlying combinatorics}
   \author{ Patryk Pagacz, Micha\l{} Wojtylak}
   \address{
   Faculty of Mathematics and Computer Science\\
   Jagiellonian University\\
   \L ojasiewicza 6\\
   30-348 Krak\'ow\\
   Poland}
   \email{patryk.pagacz@gmail.com, michal.wojtylak@gmail.com
}

\begin{abstract}
An expansion of the Weyl function of a $H$-selfadjoint random matrix with one negative square is provided. It is shown that the coefficients converge to a certain generalization of Catlan numbers. Properties of this generalization are studied, in particular,  a combinatorial interpretation is given.
\end{abstract}
\thanks{ The second author was supported by the Polish Ministry of Science and Higher Education with a Iuventus Plus grant IP2011 061371.}

\maketitle

\section*{Introduction}

The main object of the present paper are the limit properties of the following product  of a random and deterministic matrices
$$
X_N^{(d)}=H_N^{(d)}W_N,
$$
where $W_N$ stands for the Wigner matrix,
$$
 H^{(d)}_{N}=\matp{d & 0 \\ 0 & I_{N}}
$$
 and $d$ is a nonrandom real parameter. Note that for $d\ge0$ the matrix $X_N^{(d)}$ is similar to the symmetric matrix
 $$
 H_N^{(\sqrt d)}W_N H_N^{(\sqrt d)},
 $$
and hence, it has real spectrum only. The statistical properties of perturbed  symmetric random matrices were studied in many papers e.g. \cite{BG1,BG2,BG3,cap1,capitaine,FP,knowles} and the theory is well established. However, to our knowledge there is very less known about the spectral properties of $X_N^{(d)}$ in the case $d<0$.
Note that in this case $X_N^{(d)}$ is  $H_N^{(1/d)}$-selfadjoint, which allows to apply the indefinite linear algebra theory. We refer the reader to \cite{GLR} as a basic reference for $H$-selfadjoint matrices and to \cite{Wojtylak12b}, where the indefinite linear algebra was linked with randomness.

To track the (possibly complex)  eigenvalues of the matrix $X^{(d)}_N$ the Weyl function
 $$
Q_{(d)}^N (z)=-\left(e_0^*H_N^{(d)}\left(X^{(d)}_N-z\right)^{-1}e_0\right)^{-1}=-\frac1d \left(e_0^*\left(X^{(d)}_N-z\right)^{-1}e_0\right)^{-1}.
$$
is considered. The zeros of $Q_{(d)}^N(z)$ are necessarily eigenvalues of $X^{(d)}_N$ by the Schur complement reasoning.
This approach corresponds to the  technique of Weyl function, or $m$--function in operator theory, see e.g. \cite{DHS1, DHS3, KL71,KL77,KL81,SWW} for the applications in the indefinite inner product setting.

Note that as $X_N^{(d)}$ is a random matrix, the function  $Q_{(d)}^N(z)$ is random as well. The main object  of the present paper are  properties  of the function $Q_{(d)}^N(z)$ for large matrices. In particular, Theorem \ref{probexpansion}
says that $Q_{(d)}^N(z)$ converges in probability with $N\to\infty$ to
 $$
 Q_{(d)}(z)= \hat\sigma(z)+\frac{z}{d}= \frac{(2-d)z+d\sqrt{z^2-4}}{2d},
$$
 where $\hat\mu(z)$ denotes the Stieltjes transform of a measure $\mu$ and $\sigma$ the Wigner semicircle measure.
This allows to determine the limit behavior of the eigenvalues of $X_N^{(d)}$ outside $[-2,2]$ as a functions of the real parameter $d$. Namely, for $d\in (-\infty,0)$ the equation $Q_{(d)}(z)=0$ has two complex solutions
 $$
 z_{(d)}^\pm=\pm \frac{d}{\sqrt{1-d}}\ii.
 $$
 For $d\in[0,2]$ the equation $Q_{(d)}(z)=0$ has no solutions outside $[-2,2]$. For $d\in(2,\infty)$ the equation $Q_{(d)}(z)=0$ has two real solutions outside $[-2,2]$
$$
z_{(d)}^\pm=\pm \frac{d}{\sqrt{d-1}}.
$$
Summarizing, it follows that for $d\in(-\infty,0)\cup(2,\infty)$ there are precisely two eigenvalues of $X_{(d)}^N$ with limits in $\Comp\setminus[-2,2]$, the limits being  $z_{(d)}^\pm$, see Theorem \ref{probexpansion} for details.

 The idea of the present paper is to provide a combinatorial interpretation of $Q_{(d)}(z)$ in the spirit of the original Wigner's calculations. This includes, in particular, providing the following expansion
$$
\frac{-1}{ Q_{(d)}(z)}=-\frac{d}{z}\sum\limits_{n=0}^{\infty}\frac{1}{z^{2n}} \pi^{(d)}_n.
 $$
The numbers $\pi_n^{(d)}$ are a generalization of the Catalan numbers and have a natural interpretation in terms of Dyck paths or noncrossing partitions. They appear   in the study of the $t$--transformation of a measure or a free convolution  \cite{bozjwys98,bozjwys01} and deformation of free Gaussian random variables \cite{jwys05,jwys05b}. This issue is further discussed in Section \ref{S1} and in the closing remarks.

The paper is organized in a reverse order, compared to the presentation above.
In Section \ref{S1} we  define the numbers $\pi^{(d)}_n$ and show their basic properties: generating function, relation to Catalan numbers and  closed formulas as polynomials in $d$.
Section \ref{S2} is devoted to computing the limit in probability of the moments of the function  $-1/Q_N^{(d)}$. Namely, it is shown in Theorem \ref{central} that $e_0^*\left(X_N^{(d)}\right)^ne_0$
 converges to zero if $n$ is odd and to $\pi^{(d)}_{n/2}$ if $n$ is even. The result is then used in Section \ref{S3} to prove  the aforementioned Theorem \ref{probexpansion} on the limit of $Q_{(d)}^N$  and the behavior of the eigenvalues of $X_N^{(d)}$ for large $N$. In the last section we discuss the limitations of this method of study of spectra of $H$-selfadjoint random matrices. The main results of the paper are Theorem \ref{central} and Theorem \ref{probexpansion}.

The authors are grateful to Marek Bo\.zejko, Anna Kula and Janusz Wysocza\'nski for inspiring discussions and  helpful  comments.

\section{The numbers $\pi^{(d)}_n$ and their relation to Catalans }\label{S1}


A \emph{Dyck path} of order $n$ is a walk from $(0,0)$ to $(2n,0)$ in the upper-half plane,  consisting of vectors $[1,1]$ and $[1,-1]$. The set of all Dyck paths of order $n$ will be denoted by $D_n$. A classical result says that
the cardinality $c_n$ of $D_n$ satisfies
\begin{equation}\label{Catalandef}
c_n= \frac{1}{n+1}\binom{2n}{n},
\end{equation}
which can be  seen by the  recurrence formula
\begin{equation}
 \label{Catalanrec}
 c_0=1,\quad c_n = \sum_{j=1}^{n}  c_{j-1} c_{n-j}, \quad n=1,2,\dots.
\end{equation}
The numbers $c_n$ are called the \textit{Catalan numbers}.

If $w$ is a Dyck path then by $\xi(w)$ we denote the number of meetings of $w$ with the $x$-axes, excluding the point $(2n,0)$.  We define
\begin{equation}
 \label{pidef}
\pi^{(d)}_0=1,\quad \pi^{(d)}_n=\sum_{w\in D_n} d^{\xi(w)}, \quad n=1,2,\dots.
\end{equation}
Clearly, $\pi^{(1)}_n=c_n$, for $\pi_n^{(-1)}$ see the end of this section.

The numbers $\pi_n^{(d)}$ can be interpreted in the language of non-crossing 2-partitions as follow.
Let $\mathcal{NC}_2(2n)$ denote the set of all non-crossing partitions with two-element blocks only. For a non-crossing partition
$\nu\in \mathcal{NC}_2(2n)$ with blocks $\nu = \{B_1,B_2, . . .,B_n\}$ a block $B_j = \{s_1, s_2\}$ is called
\emph{outer} if there is no block $B_i=\{s_3,s_4\}$ such that   $s_3 < s_1, s_2 < s_4$. Blocks which do not enjoy this property are called \emph{inner}.
By $\sharp \ou(\nu)$ ($\sharp \inn(\nu)$) we denote the number of outer (resp. inner) blocks of a partition $\nu$.
Using this notations we can write
\begin{equation}
\label{piout}
\pi_n^{(d)}=\sum_{\nu\in \mathcal{NC}_2(2n)} d^{\sharp\ou(\nu)}.
\end{equation}
In \cite{bozjwys98} the numbers
\begin{equation}
\label{C_n(d)}
C_n(d)=\sum_{\nu\in \mathcal{NC}_2(2n)} d^{\sharp \inn(\nu)}
\end{equation}
were considered as moments of the central limit measure for the $t$-transformed free convolution, see \cite{BSp96} for a  general form of a non-commutative central limit theorem and  \cite{bozjwys01} for other examples.
One sees that both $\pi_n^{(d)}$ and $C_n(d)$ are polynomials in $d$. Furthermore, by $\inn(w)+\ou(w)=n$ one has
\begin{equation}\label{rev}
\pi_n^{(d)}=\rev(C_n(d)),
\end{equation}
 where
$\rev(\sum_{i=0}^{n}a_id^i)=\sum_{i=0}^{n}a_id^{n-i}$.

The basic recurrence relation for the numbers $\pi^{(d)}_n$ is the following.
\begin{lemma}
For $d\in\Real\setminus\set{0}$
\begin{equation}\label{rekPI}
 \pi^{(d)}_n = d\sum_{k=1}^{n}  c_{k-1} \pi^{(d)}_{n-k},\quad n=1,2\dots.
\end{equation}
\end{lemma}

\begin{proof}
We use a standard argument of considering the first intersection with the $x$--axes. Each Dyck path $w\in D_n$  is uniquely determined as a concatenation    a Dyck path $w_1$ of order $k\leq n$ with $\xi(w_1)=1$ and a Dyck path $w_2$ form $(2k,0)$ to $(2n,0)$.
The path $w_1$ consists of the vector $[1,1]$, some Dyck path $w_1'$ from $(1,1)$ to $(2k-1,1)$ of order $k-1$ and the vector $[1,-1]$.
The Dyck paths $w_1'\in D_{k-1}$ and $w_2\in D_{n-k}$ uniquely determine $w$.  
Hence
%
 $$
\pi^{(d)}_n=\sum_{k=1}^n \sum_{ w'_1\in D_{k-1}} \sum_{w_2\in D_{n-k}} d\cdot d^{\xi(w_2)} =
d\sum_{k=1}^{n}  c_{k-1} \pi^{(d)}_{n-k}.
$$

\end{proof}

The following result can be easily obtained from the results in  \cite{jwys05}, where similar calculations were derived for $d$ in the operator theory setting,  see also \cite{jwys05b} for generalizations.
For the completness of the presentation we include an elementary proof.

\begin{proposition}\label{gener}
Let $d\in\Real\setminus\set0$. The generating function $G_{(d)}(z)$ of the sequence $\big(\pi^{(d)}_n\big)_{n=0}^\infty$
satisfies
$$
G_{(d)}(z)=\frac{1}{1-zd F(z)}=\frac{2}{2-d+d\sqrt{1-4z}},
$$
where $F(z)=\frac{1-\sqrt{1-4z}}{2z}$ is the generating function for the Catalan numbers.
\end{proposition}

\begin{proof}
Using the formula \eqref{rekPI} we obtain
$$
G_{(d)}(z)-1=\sum_{n=1}^{\infty}\pi^{(d)}_nz^n=d\sum\limits_{n=1}^{\infty}\sum\limits_{k=1}^{n}\pi^{(d)}_{n-k}c_{k-1}z^{n-1}z = dG_{(d)}(z)F(z)z.
$$

\end{proof}

Using the Lemma 1 we can show a simpler recurrence formula for $\pi_n^{(d)}$.

\begin{lemma}\label{recpi2L}
For $d\in\Real\setminus\set{0,1}$ one has
\begin{equation}\label{rekPI2}
\pi^{(d)}_n= \frac{-d^2}{1-d}\pi^{(d)}_{n-1}+\frac{d}{1-d}c_{n-1}, \quad n = 1,2\dots.
\end{equation}
\end{lemma}

\begin{proof}
First observe  that
$$
\frac{-d^2}{1-d}\pi^{(d)}_{0}+\frac{d}{1-d}c_{0}=\frac{-d^2}{1-d}+\frac{d}{1-d}=d=\pi^{(d)}_1 .
$$
Now let us assume that \eqref{rekPI2} holds for all  $j\le n$, where $n\geq1$ is fixed. Then

$$
\pi^{(d)}_{n+1}= d\sum\limits_{j=0}^{n} c_{n-j}\pi^{(d)}_j = d\sum\limits_{j=1}^{n} c_{n-j}\pi^{(d)}_j + dc_n =
$$ $$
 d\sum\limits_{j=1}^{n} c_{n-j}\Big(\frac{-d^2}{1-d}\pi_{j-1}^{(d)}+\frac{d}{1-d}c_{j-1}\Big) +d c_n =
$$ $$
 \frac{-d^2}{1-d}d\sum\limits_{j=1}^{n} c_{n-j} \pi_{j-1}^{(d)} + d\frac{d}{1-d}\sum\limits_{j=1}^{n} c_{n-j}c_{j-1} +d c_n =
$$ $$
\frac{-d^2}{1-d} \pi^{(d)}_n + d\frac{d}{1-d}c_n+dc_n = \frac{-d^2}{1-d}\pi^{(d)}_n+\frac{d}{1-d}c_n.
$$

\end{proof}

The following proposition shows a closed formula for $\pi_n^{(d)}$.
The Catalan triangle $[t_{n,k}]_{n,k=0}^\infty$ (see  \cite[A009766]{OEIS}) is defined as
$$
t_{0,k}=\delta_{0,k},\quad t_{n,k}=\sum_{j=0}^k t_{n-1,j}=
\left( \binom{n+k}{k+1}-\binom{n+k}{k}\right).
$$

\begin{proposition}
The numbers $\pi_n^{(d)}$ satisfy
$$
\pi_0^{(d)}=1,\quad \pi_n^{(d)}=\sum_{k=1}^n t_{n-1,n-k} d^{k}\ (n>0).
$$
\end{proposition}

The proof can be obtained from the relation  \eqref{rev} and the formula
$$
C_n(d) = 1+
\sum^{n-2}_{k=0} d^{k+1}t_{n,k},
$$
see Proposition 6.1 in \cite{bozjwys01}. However, we present a simple argument using Lemma \ref{recpi2L}.
\begin{proof}
By formula \eqref{rekPI} it is clear that
$$
\pi_n^{(d)}=\sum_{k=1}^n a_{n-1,n-k} d^{k}\ (n>0),
$$
with some coefficients $a_{n,k}$. Since $\pi_{n-1}^{(1)}=c_{n-1}$, we have
$$
\sum_{k=1}^{n-1} a_{n-2,n-1-k}=c_{n-1},\quad n>1.
$$
Using this and Lemma \ref{recpi2L} we obtain for $n>1$
$$
\pi_n^{(d)} = \frac{-d^2}{1-d}\pi^{(d)}_{n-1}+\frac{d}{1-d}c_{n-1}=
$$ $$
 \frac{-d^2}{1-d}\sum_{k=1}^{n-1} a_{n-2,n-1-k} d^{k}+\frac{d}{1-d}c_{n-1}=
$$ $$
\sum_{k=1}^{n-1} \left(\frac{-d^{k+2}}{1-d} +\frac{d}{1-d} \right)   a_{n-2,n-1-k} =
$$  $$
\sum_{k=1}^{n-1} \left( d+ d^2 +\cdots +d^{k+1}\right)   a_{n-2,n-1-k}.
$$
Comparing the coefficients of polynomials on the left and right hand side of the above we obtain for $l=1\dts n$
$$
a_{n-1,n-l}=\sum_{k=l-1}^{n-1} a_{n-2,n-1-k}=\sum_{j=0}^{n-l} a_{n-2,j}.
$$
This, together with the information that $\pi_1^{(d)}=d$, proves the result.
\end{proof}

To finish the section, let us show that   $\pi^{(-1)}_n=(-1)^na_n$ ($n=0,1\dots$), where
 $$
a_n=\left(\frac12\right)^n\left(1+\sum\limits_{k=0}^{n-1} c_k(-2)^k\right) ,\quad  n=0,1\dots,
$$
 are the generalized Catalan numbers $C(-1,n)$ (see \cite{BH} and the OEIS database \cite{OEIS} number A064310).
Indeed, for $d=-1$ the formula \eqref{rekPI2} is of the form $$\pi^{(-1)}_n= -\frac{\pi^{(-1)}_{n-1}+c_{n-1}}{2}.$$
On the other hand,  we have
 $$
(-1)^na_n=-\frac12\left(\left(-\frac12\right)^{n-1}\left(1+\sum\limits_{k=0}^{n-2} c_k(-2)^k\right)
+\left(-\frac12\right)^{n-1}c_{n-1}(-2)^{n-1}\right)=
$$ $$
-\frac{(-1)^{n-1}a_{n-1}+c_{n-1}}{2}.
$$
Finally, to conclude $\pi_n^{(-1)}=(-1)^na_n$  it is enough to see that $\pi_0^{(-1)}=1=a(0)$.

\section{Wigner matrices with one negative square}\label{S2}

In what follows
$$
W_N=\frac{1}{\sqrt{N}}[x_{ij}]_{ij=0}^N
$$
stands for the Wigner matrix, that is a random matrix with  entries $x_{ij}$ ($0\leq i \leq j \leq N$)  being  real, independent, zero mean, the off-diagonal entries $x_{ij}$ ($0\leq i < j \leq N$) being identically distributed, and the diagonal entries  $x_{ii}$ ($0\leq i \leq N$) being identically distributed. For simplicity we assume that the variance of the off-diagonal entries is one. Moreover, we assume that
\begin{equation}\label{moments}
 r_k:=\max\set{|x_{00}|^k,|x_{11}|^k } < +\infty
\end{equation}
We set
$$
H^{(d)}_{N}=\matp{d & 0 \\ 0 & I_{N}},\quad X_N^{(d)}=H_N^{(d)}W_N,
$$
where $d$ is a nonrandom, nonzero, real parameter. 

\begin{theorem}\label{central}
Let the random matrix $X_N^{(d)}$ satisfy the above probabilistic assumptions. Then, for $d\in\Real\setminus\set0$
 $$
e_0^*\left(X_N^{(d)}\right)^{n}e_0 \to \begin{cases} \pi^{(d)}_{n/2}  &:\  n\text{ \rm even }\\
                                                                   0  &:\  n\text{ \rm  odd }
                                       \end{cases}
\quad (N\to \infty)
$$ 
in $L^2(\mathbb{P})$ and, in particular, in probability.
\end{theorem}

A basis for our considerations is the combinatorial proof of the Wigner's result, mainly as presented in \cite{AGZ}. Before the proof of Theorem \ref{central} we review the classical proof, also  introducing notations needed later on. Lemma 2.1.6 of \cite{AGZ} shows that
$$
\Ee \frac{\tr (W_N)^n}N \to \begin{cases} c_{n/2}  &:\  n\text{ \rm even }\\
                                              0  &:\  n\text{ \rm  odd }
                                       \end{cases}
\quad (N\to \infty).
$$
The proof is based on passage to the limit in the formula
\begin{equation}\label{EE}
\Ee \frac{\tr (W_N)^n}N =\frac 1N\sum_{ i_1,i_2,...,i_n=0}^N\Ee W_N(i_1,i_2)W_N(i_2,i_3)...W_N(i_{n},i_1).
\end{equation}

To analyze the above expression the set $\mathcal{W}$ of all closed words $(i_1\dts i_n,i_1)$ over the alphabet $\set{1\dts N}$ is introduced.  On this set an equivalence relation $\sim$ is defined by saying that two words are equivalent if there exists a bijection on the alphabet that maps one word to the other. We denote the set of all equivalence classes by $[\mathcal{W}]_\sim$. The \textit{weight} of a word is the number of its distinct letters. Note that all words in one equivalence class $A\in[\mathcal{W}]_\sim$ have the same weight, and the number of elements of the class equals
$$
C_{N,t}=N(N-1)\cdots(N-t+1)=\mathcal{O}(N^t),
$$
where $t$ is the common weight of the words.

Observe that we can rewrite the right hand side of \eqref{EE} as

\begin{equation}\label{EWN}
\frac 1N\sum_{A\in [\mathcal{W}]_\sim} \sum_{ (i_1,i_2,...,i_n,i_1)\in A} \Ee W_N(i_1,i_2)W_N(i_2,i_3)...W_N(i_{n},i_1).
\end{equation}

It is shown in \cite[Subsection 2.1.3]{AGZ} that the number of equivalence classes is  bounded from above for all $N$. In consequence, the limit of \eqref{EWN} is zero for $n$ odd. For $n$ even the passage to infinity with $N$ in formula \eqref{EWN} is survived only by
the equivalence classes that are in one-to-one correspondence with Dyck paths, the canonical correspondence is described in \cite[p.15]{AGZ}. Each of such classes contains words of weight $n/2+1$ and consequently its power equals $C_{N,n/2+1}=\mathcal{O}(N^{n/2+1})$.
In consequence,
\begin{equation}\label{ND}
\lim_{N\to\infty} \Ee \frac{\tr (W_N)^n}N =  \lim_{N\to\infty} \frac1N\frac{1}{N^{n/2}}\sum_{w\in D_{n/2}} C_{N,n/2+1}  E(x_{1,2}^2)^{n/2}=
\end{equation}
$$
 \sum_{w\in D_{n/2}} 1 = c_{n/2}.
$$
Having this classical argument recalled,  we are ready to prove Theorem \ref{central}.

\begin{proof}[Proof of Theorem \ref{central}] The proof consists of three usual steps.

\textit{Step 1.}
$$
\Ee e_0^*\left(X_N^{(d)}\right)^{n}e_0 \to
\left\{
\begin{array}{rl}
\pi^{(d)}_{\frac{n}{2}} \textnormal{ , for n even, } \\
 0 \textnormal{ , for n odd}. \\
\end{array}
\right.
 $$
We begin the proof with an analogue of \eqref{EE}
$$
\Ee\left( e_0^* (X_N)^n e_0\right)  =\sum_{ i_1,i_2,...,i_{n-1}=0}^N\Ee X_N(0,i_1)X_N(i_1,i_2)...X_N(i_{n-1},0)
$$

 $$
=\sum_{ i_1,i_2,...,i_{n-1}=0}^N d^{\eta(0,i_1\dts i_{n-1})} \Ee W_N(0,i_1)W_N(i_1,i_2)...W_N(i_{n-1},0),
$$
where $\eta(j_1\dts j_k)$ is the number of zeros in the sequence $(j_1\dts j_k)$.
We introduce $\mathcal{ W}_0$ as the set of words over $\set{0\dts N}$ of the form $(0,i_1\dts i_{n-1},0)$.
Note that all words in one equivalence class $A\in[\mathcal{W}_0]_\sim$ have the same weight, and the number of elements of the class equals $C_{N,t-1}$, where $t$ is the common weight of the words.
Then, analogously to \eqref{EWN}, one has
\begin{equation}\label{EWN2}
\Ee e_0^*\left(X_N^{(d)}\right)^{n}e_0 =
\end{equation}
$$
\sum_{A\in [\mathcal{ W}_0]_\sim} \sum_{ (0,i_1,i_2,...,i_{n-1},0)\in A}d^{\eta(0,i_1\dts i_{n-1})} \Ee \left(W_N(0,i_1)W_N(i_1,i_2)...W_N(i_{n-1},0)\right).
$$
As in the proof of the Wigners result, the limit in the case $n$ is odd is zero, and in the  case $n$ is even the passage to infinity with $N$ in formula \eqref{EWN2} is survived only by
the equivalence classes that are in one-to-one correspondence with Dyck paths of order $n/2$. Each of such classes contains words of weight $n/2+1$ and consequently has precisely  $C_{N,n/2}$ elements.
Hence,
\begin{equation}\label{NDxi2}
\lim_{N\to\infty} \Ee e_0^*\left(X_N^{(d)}\right)^{n}e_0 =   \lim_{N\to\infty} \frac{1}{N^{n/2}}\sum_{w\in D_{n/2}} d^{\xi(w)}C_{N,n/2}  \Ee(x_{1,2}^2)^{n/2}.
\end{equation}
We used in the above equality the fact, that for all words $(0,i_1\dts i_{n-1},0)$ from class $A\in[\mathcal{W}_0]_\sim$ we have
\begin{equation}\label{etaxi}
\eta(0,i_1\dts i_{n-1})=\xi(w),
\end{equation}
where $w$ is the Dyck word corresponding to the class $A$. Indeed, in the canonical bijection described in \cite{AGZ} meeting of the Dyck path with the $x$-axes corresponds to a zero on the corresponding position in the word. Finally, by \eqref{NDxi2},
$$
\lim_{N\to\infty} \Ee e_0^*\left(X_N^{(d)}\right)^{n}e_0 =  \sum_{w\in D_{n/2}} d^{\xi(w)}=\pi_{n/2}.
$$

\textit{Step 2.}
$$
\Ee \left(e_0^*\left(X_N^{(d)}\right)^{n}e_0\right)^2 \to
\left\{
\begin{array}{rl}
 \big(\pi^{(d)}_{\frac{n}{2}}\big)^2 \textnormal{ , for n even, } \\
 0 \textnormal{ , for n odd}. \\
\end{array}
\right.
 $$
We start the proof similarly as in Step 1
\begin{eqnarray*}
& \Ee\left( e_0^* (W_N)^n e_0\right)^2  =\\
& \sum\limits_{ i_1,\dots,i_{n-1}=0}^N\sum\limits_{ j_1,\dots,j_{n-1}=0}^N\Ee \left( X_N(0,i_1)\cdots X_N(i_{n-1},0)   X_N(0,j_1)\cdots X_N(j_{n-1},0)\right)
\\
&=\sum\limits_{ i_1,\dots,i_{n-1}=0}^N\sum\limits_{ j_1,\dots,j_{n-1}=0}^N d^{\eta(0,i_1\dts i_{n-1},0,j_1\dts j_{n-1})}\\&
 \cdot\Ee  \left(W_N(0,i_1)\cdots W_N(i_{n-1},0)W_N(0,j_1)\cdots W_N(j_{n-1},0)\right).
\end{eqnarray*}
We introduce $\mathcal{ W}_{00}$ as the set of words over $\set{0\dts N}$ of the form
$$
(0,i_1\dts i_{n-1},0,j_1\dts j_{n-1},0).
$$
Note that the power of each equivalence class  $A\in[\mathcal{W}_{00}]_\sim$ equals $C_{N,t-1}$, where $t$ is the common weight of the words in $A$.
Furthermore,
\begin{equation}\label{EWN3}
\Ee \left(e_0^*\left(X_N^{(d)}\right)^{n}e_0\right)^2 =
\end{equation}
$$
\sum_{A\in [\mathcal{ W}_{00}]_\sim} \sum_{ I\in A}d^{\eta(0,i_1\dts i_{n-1},0,j_1\dts j_{n-1})}
$$
$$
\cdot \Ee  \left(W_N(0,i_1)\cdots W_N(i_{n-1},0)W_N(0,j_1)\cdots W_N(j_{n-1},0)\right),
$$
with
$$
I=(0,i_1\dts i_{n-1},0,j_1\dts j_{n-1},0).
$$
The passage with $N$ to infinity is again survived by the equivalence classes corresponding to Dyck paths $D^0_n$ of order $n$, which meet the $x$ axes after $n$ steps. Since there is no word in $\mathcal{W}_{00}$ of  length $2n+1$ corresponding to a Dyck path,   the limit with $N\to\infty$ of the expression above is zero for $n$ odd. For $n$ even the words in equivalence classes that survive passage $N\to\infty$ have weight $n+1$. Consequently
\begin{equation}\label{NDxi3}
\lim_{N\to\infty}\Ee \left(e_0^*\left(X_N^{(d)}\right)^{n}e_0\right)^2 =
\end{equation}
$$
  \lim_{N\to\infty} \frac{1}{N^{n}}\sum_{w\in D_{n}^0} d^{\xi(w)}C_{N,n}  E(x_{1,2}^2)^{n}=\sum_{w\in D_{n}^0} d^{\xi(w)}.
$$
Note that each Dyck path  $w\in D^0_n$ is a concatenation of two Dyck paths $w_1,w_2\in D_{n/2}$ and $\xi(w)=\xi(w_1)+\xi(w_2)$. Hence,
\begin{equation}\label{NDxi3-1}
\lim_{N\to\infty}\Ee \left(e_0^*\left(X_N^{(d)}\right)^{n}e_0\right)^2 =
\end{equation}
$$
 \sum_{w_1,w_2\in D_{n/2}} d^{\xi(w_1)+\xi(w_2)}=\left(\pi^{(d)}_{n/2}\right)^2.
$$
\textit{Step 3.}
The application  of Chebyshev's inequality finishes the proof.

\end{proof}


\section{Representation and convergence   of the Weyl functions}\label{S3}
As it was mentioned in the Introduction, we define the Weyl function as
\begin{equation}\label{QNd}
Q_{(d)}^N(z)=-\left(e_0^*H_N^{(d)}\left(X^{(d)}_N-z\right)^{-1}e_0\right)^{-1},
\end{equation}
where $e_0$ is the first vector of the canonical basis of $\Comp^{N+1}$ and $H_N^{(d)}$, $X_N^{(d)}$ are defined as in Section \ref{S2}. It is clear that the zeros of $Q_{(d)}^N(z)$ are eigenvalues of
 $X_N^{(d)}$, the converse is not necessarily true. Nevertheless, the Weyl function will allow us to determine the part of the spectrum of $X_N^{(d)}$ lying in $\Comp\setminus[-2,2]$ for large values of $N$.
For this aim we prove  that ${Q_{(d)}^N(z)}$ converges in  probability to the function
\begin{equation}\label{Qd}
{Q_{(d)}(z)}:= \hat\sigma(z) + \frac zd=\frac{(2-d)z+d\sqrt{z^2-4}}{2d},
\end{equation}
where $\sigma$ is  the Wigner semicircle measure.

The function $\frac{-1}{Q_{(d)}(z)}$ is for $d>0$ an ordinary Nevanlinna function and for $d<0$ a generalized Nevanlinna function with one negative square.
In both cases the general theory (see e.g. \cite{AG,KL79}) admits an expansion at infinity. More precisely,  we have  the following.

\begin{lemma}\label{lmexp}
For $Q_{(d)}(z)$ defined by \eqref{Qd} one has
\begin{equation}\label{expQ}
\frac{-1}{Q_{(d)}(z)}=-\frac{d}{z} \sum_{n=0}^\infty \frac{1}{z^{2n}}\pi_n^{(d)},
\end{equation}
where the series absolutely converges for $z\in\Comp$ such that
$$
|z|>\begin{cases} 2  &: |d|<2 \\  \frac{|d|}{\sqrt{|d-1|}}   &: |d|\geq 2  \end{cases}.
$$
\end{lemma}

\begin{proof}
Note that by Proposition \ref{gener} the right hand side of \eqref{expQ} equals
$$
-\frac{d}{z} G_{(d)}\left( \frac1{z^2}\right)=\frac{-2d}{(2-d)z+d\sqrt{z^2-4}}=\frac{-1}{Q_{(d)}(z)}.
$$
The claim on convergence follows from the fact that the function $\frac{-1}{Q_{(d)}(z)}$ is holomorphic in $\Comp\setminus\left([-2,2]\cup\{z^\pm_{(d)}\}\right)$ for
$$
z^{\pm}_{(d)}=\begin{cases} \pm \frac{d}{\sqrt{1-d}}\ii  &: d<0  \\  \pm \frac d{\sqrt{d-1}}   &: d > 2  \\ 0 &: d\in [0,2]  \end{cases},
$$
see Introduction.
\end{proof}

We  formulate now the second main theorem of the paper.  Statement (i) is a new result,  (ii), (iii) and (iv) were already proved in \cite{Wojtylak12b} with a different method that allowed to 
omitt the assumption \eqref{rk}.  A perturbation problem, similar to the one discribed in (v),  is widely discussed in the literature, see e.g. \cite{BG1,BG2,BG3,cap1,capitaine,FP,knowles}. Nevertheless, (v)  is stated  for completness of the analysis of the change of the spectrum of $X_N^{(d)}$ with the parameter $d$. 

\begin{theorem}\label{probexpansion}
Assume that $X_N^{(d)}$ satisfies the probabilistic assumption of Section \ref{S2}, then
\begin{itemize}
\item[(i)] the  function ${Q_{(d)}^N(z)}$ defined by \eqref{QNd} has the following representation
$$
{Q_{(d)}^N(z)}=-a_N+\frac zd+\hat\mu_N(z),
$$
where $a_N$ is a real, random variable, $a_N\to 0$ with $N\to 0$ in probability, $\mu_N$ is a random, discrete, probablility measure on $\Real$ and 
$$
\int_{\Real} t^n d\mu_N(t)\to \int_{\Real} t^n d\sigma(t),\quad n=0,1,\dots,
$$
with $N\to \infty$ in probability.
\end{itemize}
If, additionally, the moments $r_k$ defined by \eqref{moments} satisfy 
\begin{equation}\label{rk}
r_k\leq k^{Ck},\quad k=1,2,\dots,
\end{equation}
 for some constant $C\geq 0$  then 
\begin{itemize}
\item[(ii)] $\mu_N\to\sigma$ with $N\to\infty$ weakly in probability;
\item[(iii)] for $z\in\Comp^+$ the number ${Q_{(d)}^N(z)}$
 converges in probability with $N\to \infty$   to
${Q_{(d)}(z)}$;
\item[(iv)] if $d<0$  the $($unique$)$ eigenvalue of nonpositive type of $X_N^{(d)}$ converges in probability to $z_{(d)}^+$;
\item[(v)] if $d>2$  then the minimal $($maximal $)$ eigenvalue of $X_N^{(d)}$ converges with $N\to\infty$ to $z_{(d)}^- (z_{(d)}^+$, respectively $)$. 
\end{itemize}
\end{theorem}

\begin{proof}(i) Writing $X^{(d)}_N$ as
$$
\begin{bmatrix}
   da_N & db_N^* \\
   b_N & C_N\\
\end{bmatrix}$$
and using the Schur complement argument we see, that
$$
Q_{(d)}^N(z)=-\left(e_0^*H_N^{(d)}(X_N^{(d)}-z)^{-1}e_0\right)^{-1}= -a_N+\frac zd+b_N^*(C_N-z)^{-1}b_N.
$$
Observe that with a discrete,  random measure 
$$
\mu_N: = \sum_{j=1}^N  | f_{j}^N |^2 \delta_{\lambda^N_j},
$$
where $U_NC_NU_N^*=\diag(\lambda^N_1\dts \lambda^N_j)$ is the unitary diagonalization of $C_N$ and  $f_N=[f^N_1\dts f^N_N]^\top=U_Nb_N$, one has
\begin{equation}\label{Qdform}
Q_{(d)}^N(z)=\frac zd-a_N+\hat\mu_N(z).
\end{equation}
Expanding the Stieltjes transform of $\hat\mu(z)$ at infinity 
$$
\hat\mu_N(z)=\sum_{n=1}^\infty \alpha_n^N z^{-n}
$$ 
with (random) real coefficients 
$$
\alpha_n^N=\int t^{n-1} d\mu_N(t),\quad n=1,2\dots.
$$
Consequently, we obtain the following Laurent series expansion at infinity of $Q_{(d)}^N(z)$
 \begin{equation}\label{Lexp}
Q_{(d)}^N(z)=\sum_{n=-1}^\infty \alpha_n^N{z^{-n}},
\end{equation}
with $\alpha_{-1}^N:=\frac1d$, $\alpha_0^N:=-a_N$ ($N=1,2,\dots$). 
 On the other hand consider the expansion of $-1/Q^N_{(d)}(z)$
\begin{equation}\label{Linvexp}
-1/Q^N_{(d)}(z)=d\ e_0^*\left(X^{(d)}_N-z\right)^{-1}e_0=
 \sum_{n=1}^\infty \gamma_n^N  {z^{-n}},
\end{equation}
with
$$
\gamma_n^N:=-d\ e_0^* (X^{(d)}_N)^{n-1} e_0 .
$$
Observe that, by the Cauchy product rule, the random sequences $(\alpha_n^N)_{n=-1}^\infty$, $(\gamma_n^N)_{n=1}^\infty$  satisfy surely for each $N=0,1,\dots$ the following equalities 
\begin{equation}\label{C1}
\alpha_{-1}^N\gamma_1^N=1,\qquad \sum_{i=0}^k \alpha_{i-1}^N \gamma_{k-i+1}^N=0\quad  k=1,2,\dots.
\end{equation}
Furthermore, the nonrandom sequences  $(\alpha_n)_{n=-1}^\infty$, $(\gamma_n)_{n=1}^\infty$ defined by  $\alpha_{-1}:=\frac1d$, $\alpha_0:=0$
$$
 \alpha_n:= \begin{cases} c_{(n-1)/2}  &:\  n\text{ \rm odd }\\
                                                                   0  &:\  n\text{ \rm  even }
                                       \end{cases} \quad 
                                     \gamma_n:= \begin{cases} \pi^{(d)}_{(n-1)/2}  &:\  n\text{ \rm odd }\\
                                                                   0  &:\  n\text{ \rm  even }
                                       \end{cases},\quad n=1,2,\dots
$$
 satisfy by Lemma \ref{lmexp}
\begin{equation}\label{C2}
\alpha_{-1}\gamma_1=1,\qquad \sum_{i=0}^k \alpha_{i-1} \gamma_{k-i+1}=0,\quad  k=1,2\dots.
\end{equation}
By Theorem \ref{central} , for each $n=0,1,\dots$ 
\begin{equation}\label{gammaN}
\gamma_n^N \to\gamma_n\ (N\to\infty),\ \text{in probability}.
\end{equation} 
Emploing \eqref{C1}, \eqref{C2} and  \eqref{gammaN} in a simple induction argument with respect to $n$  we obtain that
for each $n=0,1,\dots$ 
$$
\alpha_n^N\to  \alpha_n\ (N\to\infty),\ \text{in probability}.
$$
Since $\alpha_n=\int_\Real t^{n-1} d\sigma(t)$ the proof of (i) is complete.

(ii) Note that by \cite[Theorem 2.1.22]{AGZ} $$
\mathbb{P}(\supp \mu_N\sbs[-3,3])\to 1, \quad N\to\infty.
$$
Consequently, for each $\eps>0$ 
$$
 \lim_{N\to\infty} \mathbb{P} \left( \int_{|x|>3} |x|^k d\mu_N> \eps \right) \to0.
$$
A standard approximation argument (cf. e.g. \cite{AGZ}  formula (2.1.9) and below)
shows that $\mu_N$ converges to $\sigma$ weakly in probablity.
This, together with \eqref{Qd} and \eqref{Qdform}, finishes the proof of (ii).

Statement (iii) follows directly from (ii), statement (iv) follows from (ii) and the continuity of the eigenvalue of nonpositive type as a function of $\mu_N$ and $a_N$, see \cite{Wojtylak12b}. 
To see (v) assume $d>2$. Let $M$ denote the set of pairs $(\mu,a)$ of a positive, finite measures $\mu$ on $\Real$  with 
\begin{equation}\label{suppcond}
\supp\mu\sbs\left[\frac{ -2+z_{(d)}^-}2,\frac{ 2+z_{(d)}^+}2\right ],
\end{equation}
and $a\in\Real$ such that the equation  $a+\frac zd +\hat\mu(z)=0$  has a  solution in each of the half axes $(-\infty,\min \supp\mu)$, $(\max\supp\mu, +\infty)$. Note that these both solutions are 
are necessarily unique, we denote them by $\zeta_N^-$ and $\zeta_N^+$, respectively. 
Note that $(\sigma,0)\in M$ and that the pair $(\mu_N,a)$ belongs to $M$ for some (equivalently: for every) $a\in\Real$ if and only if  \eqref{suppcond} is satisfied.  
We endow $M$ with the topology inherited from the product of the weak and natural topology. Thanks to \eqref{suppcond} the mapping  $M\ni(\mu,a)\mapsto \zeta_N^-$ is continuos at $(\sigma,0)$.

Let $\Xi_N$ stand for the event that $\mu=\mu_N$ satisfies \eqref{suppcond}. Observe that 
$$
P\left(|\zeta_N^--z^-_{(d)}|>\eps\right) \leq P\left(\Xi_N^c\right)+P\left(\Xi_N,|\zeta_N^--z^-_{(d)}|>\eps \right),
$$
and the first summand converges with $N\to \infty$ to zero by \cite[Theorem 2.1.22]{AGZ}. The second summand converges to zero since, by (i), $(\mu_N,a_N)$ converges to $(\sigma,0)$ in $M$.
Thanks to the afore-mentioned continuity at $(\sigma,0)$,  $\zeta_N^-$ converges in probability to $z_{(d)}^-$. Analogously one obtains the convergence of $\zeta_N^+$.

\end{proof}



\section{Final remarks}\label{S4}

One should mention that infinite tridiagonal matrices of the form
\begin{equation}\label{op}
J_{(d)}=\matp{
\phantom{\ddots} & \sqrt d &&\\ \sqrt d & \phantom{\ddots} & 1 \\ & 1 &\phantom{\ddots}& 1\\ && 1 && \ddots \\ &&& \ddots}
\end{equation}
are in some sense operator analogues of the matrices $H_N^{(\sqrt d)}W_N H_N^{(\sqrt d)}$,  namely the functions $e_0^*(J_{(d)}-z)^{-1}e_0$ and $-1/Q_{(d)}(z)$ coincide. The family \eqref{op} was studied in \cite{jwys05}, see therein and  \cite{bozjwys98,bozjwys01} for a relation with
noncommutative probability and $t$--transformations of convolutions. Also the representation of the function $-1/Q_{(d)}(z)$ for $d>0$ as a Stieltjes transform of a measure was derived in \cite{jwys05}.

However, for random matrices, contrary to the operator case \eqref{op}, the function $-1/Q_{(d)}(z)$ does \textit{not} provide information about the limit of the empirical measure of eigenvalues. Namely, using the intertwining principle one may easily show that the empirical measure of spectrum of the matrix $X^{(d)}_N$ converges for all $d>0$ to the Wigner measure $\sigma$, while $-1/Q_{(d)}(z)\neq\hat \sigma(z)$.  

Now let us discuss the limitations of the  combinatorial methods in computing the spectra of $H$-selfadjoint Wigner matrices. For a fixed $k\in\Nat$ consider the following matrices
$$
H_N^{(d),k}=\begin{bmatrix}
   I_k & 0 & 0 \\
   0 & d & 0\\
   0 & 0 & I_{N-k}
\end{bmatrix}\in\Comp^{(N+1)\times(N+1)},\quad  N=k,k+1,\dots.
$$
First observe that the spectrum of $X^{(d),k}_N=H_N^{(d),k}W_N$ coincides for large $N$ with the spectrum of $X^{(d)}_N$. Indeed, the unitary matrix
$$U_{k,N}=\begin{bmatrix}
   0 & 0 & 1 & 0 \\
   0 & I_{k-1} & 0&0\\
   1 & 0 & 0 & 0 \\
   0 & 0 & 0 & I_{N-k}
\end{bmatrix}.$$
satisfies $U_{k,N}H_N^{(d)}U_{k,N}=H_N^{(d),k}$. Hence, the spectrum of $X_N^{(d),k}$ coincides with the spectrum of $H_N^{(d)}U_{k,N}W_NU_{k,N}$.
Since $U_{k,N}W_NU_{k,N}$ is again Wigner matrix, we get that the limit distributions of real eigenvalues as well as the limit of the eigenvalue of nonpositive type of $X_N^{(d),k}$ coincide with, respectively, the limit distributions of real eigenvalues and the limit of the eigenvalue of nonpositive type of $X_N^{(d)}$.

However,  for the above matrices the combinatorial interpretation given  in the  proof of Theorem \ref{central} is no longer true. Namely, repeating the proof 
it is not possible to show that after dividing into equivalence classes and passing to the limit each Dyck path $w\in D_n$ contributes precisely $d^{\xi(w)}$ to the total sum, since formula \eqref{etaxi} is no longer true.
Although for the above ensemble of matrices $X_N^{(d),k}$ this does not seem to be a large drawback, the real problem appears while considering  matrices as
$$
H_N^{(d_1\dts d_\kappa)}= \diag(d_1\dts d_\kappa)\oplus I_{N-\kappa +1}
, \quad X_N^{(d_1\dts d_\kappa)}=H_N^{(d_1\dts d_\kappa)}W_N.
$$
For the calculation of spectrum of those matrices one needs to develop a different method, the topic will be treated in a subsequent paper.

\end{document}